\documentclass[12pt]{amsart}
\usepackage[OT4]{fontenc}
\usepackage{graphics,graphicx}
\usepackage{amsmath}
\usepackage{amssymb}
\usepackage {amsthm}

\usepackage{xcolor}

\usepackage{enumerate}
\usepackage{hyperref}
\usepackage[cp1251]{inputenc}
\usepackage[all]{xy}
\usepackage{cleveref}
\usepackage{lipsum}

\newtheorem{theorem}{Theorem}
\newtheorem{proposition}{Proposition}
\newtheorem{lemma}{Lemma}
\newtheorem{corollary}{Corollary}
\newtheorem*{proposition*}{Proposition}
\newtheorem{definition}{Definition}

\newtheorem{conjecture}{Conjecture}
\newtheorem{remark}{Remark}

\newcommand\myfootnote[1]{
\renewcommand{\thefootnote}{}
\footnotetext{#1}
\def\thefootnote{\@arabic\c@footnote}
}

\makeatletter
\renewcommand{\subsection}{\@startsection{subsection}{2}{0mm}{-\baselineskip}{-5pt}{\it \bf}}
\makeatother

\title{Unital locally matrix algebras  and Steinitz numbers}
\author{Oksana Bezushchak and Bogdana Oliynyk }

\thanks{The second author was partially supported by the grant 346300 for IMPAN from the Simons Foundation and the matching 2015-2019 Polish MNiSW fund}

\begin{document}

\maketitle

\address{  Faculty of Mechanics and Mathematics,
	Kyiv National Taras Shevchenko University
	Volodymyrska, 64, Kyiv 01033, Ukraine \\ Department of Mathematics, National University
		of Kyiv-Mohyla Academy, Skovorody St. 2, Kyiv,
		04070, Ukraine                     }
                 
\email{bezusch@univ.kiev.ua, oliynyk@ukma.edu.ua}

\keywords{Keyword: locally matrix algebra, Steinitz number, diagonal embedding, Clifford algebra}               

\subjclass{2010}{Mathematics Subject Classification: 03C05, 03C60, 11E88}


\begin{abstract}
An   $F$-algebra  $A$ with unit $1$ is said to be a locally matrix algebra if an arbitrary finite collection of elements $a_1,$ $\ldots,$ $a_s $ from $ A$  lies in a subalgebra $B$ with $1$ of the algebra $A$, that is  isomorphic to a matrix algebra $M_n(F),$ $n\geq 1.$ To an arbitrary unital locally matrix algebra $A$ we assign a Steinitz number   $\mathbf{n}(A)$ and study a relationship between $\mathbf{n}(A)$ and $A$.

\end{abstract}




\section{ Introduction}

 Let $F$ be a ground  field. Throughout the  paper we consider associative $F$--algebras with $1$. We say that an algebra $A$ is a locally matrix algebra if an arbitrary finite collection of elements $a_1,$ $\ldots,$ $a_s \in A$  lies in a subalgebra $B$,  $1\in B \subset A$,  that is  isomorphic to a matrix algebra $M_n(F),$ $n\geq 1.$
 
 J.~G.~Glimm \cite{Glimm} parametrised  countable dimensional locally matrix algebras  (under the name  uniformly hyperfinite algebras) with Steinitz or supernatural numbers. O.~Bezushchak, B.~Oliynyk and V.~Sushchanskii \cite{Sushch2} extended their parametrisation to regular relation structures. The idea of diagonal embeddings  was introduced by A.~E.~Zalesskii  in \cite{Zall}. In a series of  papers     A.~A.~Baranov and A.~G.~Zhilinskii  used Steinitz numbers to classify diagonal locally simple Lie algebras of countable dimension \cite{Baranov2}, \cite{Baranov1}.
 
 In this paper we introduce a Steinitz number $\mathbf{n}(A)$ for a unital locally matrix algebra of arbitrary dimension and study the relation between  $ \mathbf{n}(A) $  and $A$. 
 
  In the section 3 we show that for  locally matrix algebras $A,$ $B$ of arbitrary dimensions universal theories of  $A,$  $B$  coincide if and only if $ \mathbf{n} (A) =\mathbf{n} (B)$. 
 
 In the section 4  we consider examples of unital  locally matrix algebras   and  find their Steinitz numbers.  Finally, in the section 4 we give example of nonisomorphic locally matrix algebras $A,$ $B$ of uncountable dimension such that $\mathbf{n} (A)=\mathbf{n} (B)=2^{\infty}.$

\section{Preliminaries}
 Let $n,m \ge 1$ and let $1 \in M_n(F) \subset M_m(F)$. Let $B \cong M_k(F)$ be the centralizer of $M_n(F)$ in $M_m(F)$. Then (see \cite{Jacobson}), $M_m \cong M_n(F) \otimes_F M_k(F)$, which implies m=nk.

Let $ \mathbb{P} $ be the set of all primes. A 
  {\it Steinitz } or {\it supernatural} number  is an infinite formal
product of the form
$$ \prod_{p\in \mathbb{P}} p^{r_p}, $$
where  $ r_p \in  \mathbb{N} \cup \{0,\infty\} $.
The product of two Steinitz numbers
$$ \prod_{p\in \mathbb{P}} p^{r_p} \ \text{ and } \  \prod_{p\in \mathbb{P}} p^{k_p}  $$
is a  Steinitz number
$$ \prod_{p\in \mathbb{P}} p^{r_p+k_p},  $$
where we assume, that $k_p \in  \mathbb{N} \cup \{0,\infty\}$,   $t+\infty=\infty+t=\infty+\infty=\infty$ for all positive integers $t$. 
Denote by $ \mathbb{SN} $ the set of all Steinitz
numbers. Obviously, the set of all positive integers $ \mathbb{N} $ is a  subset of the set of all Steinitz numbers $ \mathbb{SN} $. The elements of the set $ \mathbb{SN} \setminus
\mathbb{N} $ are called {\it infinite Steinitz } numbers.

The Steinitz number $ v $ divides $ u $ if there exists $ w\in
\mathbb{SN} $, such that $ u=v\cdot w $. The divisibility relation
$ | $ makes $ \mathbb{SN} $ into a partially ordered set with the
greatest element   $$ I = \prod_{p\in P} p^\infty $$ and the least
element $1$. Moreover, the poset $(\mathbb{SN}, |)$ is a complete
lattice.

Steinitz  numbers were introduced by Ernst Steinitz \cite{ST} in 1910 to classify algebraic extensions of finite fields.

\begin{definition}
	\label{St_number}
Let $A$ be an infinite dimensional unital locally matrix algebra over a field $F$ and let $D(A)$ be  the  set of all positive integers $n$, such that  there is a subalgebra  $A'$, $A'\subseteq A$, $A' \cong  M_n(F)$ and   $ 1_A \in A'$, where  $ 1_A $ is the identity of $A$. 
Then the  least common multiple of the set $D(A)$ is called the {\it Steinitz or supernatural number } $\mathbf{n}(A)$ of the algebra $A$. 
\end{definition}
The following proposition is straightforward. 
\begin{proposition}
	\label{tensor}
	Let $A$ and $B$ be unital locally matrix algebras. Then the algebra $A \otimes_F B$ is a unital  locally matrix and  
	$$\mathbf{n}(A) \cdot \mathbf{n}(B)= \mathbf{n}(A \otimes_F B).$$
\end{proposition}

\begin{theorem}[See \cite{Baranov1}, \cite{Sushch2}, \cite{Glimm}]
	\label{teorBOS}
If $A$ and $B$  are unital locally matrix algebras of countable dimension then $A$ and $B$ are isomorphic if and only if  $\mathbf{n}(A)=\mathbf{n}(B)$.
\end{theorem}

For the rest of this paper, we will write $M_n(F) \subseteq A$ if there is a subalgebra $A'$, $1 \in A'\subseteq A $, such that $A'\cong M_n(F)$.

\section{Universally equivalent algebras}

Let $A$ be an  algebraic system (see \cite{Malcev}). The universal elementary theory $UTh(A)$ consists of universal closed  formulas (see \cite{Malcev}) that are valid on $A$. The systems $A$ and $B$ of the same signature are universally equivalent if $UTh(A)=UTh(B)$.

\begin{theorem}
	\label{univ1}
	Let $A$ and $B$ be unital locally matrix algebras.  Then    $A$ and $B$ are   universally equivalent if and only if  their  Steinitz numbers $\mathbf{n}(A)$ and $\mathbf{n}(B)$ are equal, i.e. 
	\begin{equation*}
	\mathbf{n}(A)=\mathbf{n}(B).
	\end{equation*}
\end{theorem}

For proof of this theorem we need a lemma.
\begin{lemma}
	\label{form1}
	Let $A$ be a unital locally matrix algebra. The property $n \notin D(A)$ can be written as a universal closed  formula.
\end{lemma}

\begin{proof}
Consider the formula \begin{multline}
\label{f_1} 
\forall x_{ij}, \ 1 \le i,j \le n, \ \ ( (\land_{i,j,t,s} \ x_{ij}x_{ts}=\delta_{jt}x_{is}) \\\land(x_{11}+x_{22}+\ldots +x_{nn}=1) \rightarrow (1=0))
,\end{multline}where  $\delta_{jp}$ is the Kronecker delta, i.e. $\delta_{jp}=\begin{cases}
1, & \text{if  $j =p$ }\\
0, & \text{if  $j \neq p$}
\end{cases}$
and $x_{ij}$, $ 1 \le i,j \le n $.

The universal closed  formula \eqref{f_1} is true if and only if $M_n(F) \nsubseteq A$.
\end{proof}

\begin{proof}[Proof of Theorem \ref{univ1}]

First we recall that if some universal closed  formula is true on an algebraic structure, then it is true on any algebraic substructure.
	
If unital locally matrix algebras 	$A$ and $B$ are  universally equivalent, then by Lemma \ref{form1}  $M_n(F) \nsubseteq A$ if and only if $M_n(F) \nsubseteq B$. This implies that $D(A)$ is equal to $D(B)$ and therefore $\mathbf{n}(A)=\mathbf{n}(B)$.

Let's now suppose that  $\mathbf{n}(A)=\mathbf{n}(B)$. From the Malcev Local Theorem (see \cite{Malcev}) it follows, that the algebra $A$ is unitally embeddable in an ultraproduct of matrix algebras $M_n(F)$, $n \in D(A)$. If $\mathbf{n}(A)=\mathbf{n}(B)$,  then $D(A)=D(B)$. Since every $M_n(F)$, $n \in D(A)$, is unitally embeddable in $B$, it follows that the algebra $A$ is unitally  embeddable in an ultrapower of the algebra $B$. 

As  the algebra $A$ is unitally embeddable  in an ultrapower of the algebra $B$, if some universal  closed  formula is true on $B$, then it is true on $A$. 

Similar, the unital locally matrix algebra $B$ is unitally embeddable in an ultrapower of $A$, so if some universal  closed  formula  is true on $A$, then it is true on $B$ and our theorem is proved. 
\end{proof}

\begin{remark}
	It is important that identity is added to the signature. Without $1$ in the signature any two infinite dimensional locally matrix algebras are elementarily equivalent.   
\end{remark}

\section{Examples }
 
Let $V$ be a vector space over a field $F$, $char F \ne 2$. Firstly, we recall a construction of Clifford algebras. 

A map $f: \ V \to F$ is called a { \it quadratic form} if the  following conditions hold:
\begin{enumerate}
	\item 
	$f(\alpha v)=\alpha ^2 f(v)$ for every $\alpha\in F$ and $v \in V$;
	\item 
	$f(u,v)=f(u+v)-f(v)-f(u)$ is a bilinear form. 
\end{enumerate}

 The quadratic form $f: \ V \to F$ is { \it nondegenerate} if and only if the bilinear form $f(u,v)$ is nondegenerate.
 
The Clifford algebra $Cl(V,f)$ is a unital algebra  generated by the vector space $V$ and $1$ and defined by relations $v^2=f(v)\cdot 1$ for all $v \in V$. Hence,
$$vw+wv=f(v,w) \cdot 1 $$
holds for any $v, w$ from the vector space $V$ in the Clifford algebra $Cl(V,f)$.

Let  $\{v_i\}_{i \in I}$ be a basis of the  vector space $V$. Assume, that  the set of  indexes $I$ is ordered.  Then all possible  ordered products $v_{i_1} v_{i_2} \ldots v_{i_k}$, $i_1 < i_2 < \ldots < i_k$, and $1$ (that can be defined as the empty product) is a  basis of the  Clifford algebra $Cl(V,f)$.

Obviously, if $\dim_F V=n$, then $\dim_F Cl(V,f)=2^n$. If $V$ is an  infinite dimensional vector space, then $\dim_F V =\dim_F  Cl(V,f)$. 

Onward, we assume that the field $F$ is algebraically closed and the  quadratic form $f$ is nondegenerate. 

\begin{theorem} \cite{Jacobson}
	\label{jacob}
Let $Cl(V,f)$ be the Clifford algebra defined by a non\-de\-ge\-ne\-ra\-te quadratic form $f$ on $V$ and $\dim_F V=n< \infty$. If the number $n$ is even, then the Clifford algebra  $Cl(V,f)$ is isomorphic to the matrix algebra $M_{2^{\frac{n}{2}}}(F)$. If 
the number $n$ is odd, then the algebra  $Cl(V,f)$ is isomorphic to the direct sum of  matrix algebras $M_{2^{\frac{n-1}{2}}}(F) \oplus M_{2^{\frac{n-1}{2}}}(F)$.
\end{theorem}

\begin{theorem}
	\label{dimCl}
Let  $V$ be  an infinite dimensional vector space. Then the Clifford algebra  $Cl(V,f)$ is locally matrix and $\mathbf{n}(Cl(V,f))=2^{\infty}$.	
\end{theorem}

\begin{proof}
Assume, that $S$ is a finite subset of the Clifford algebra  $Cl(V,f)$. As $S$ is finite, there is a  finite dimensional subspace $W$ of the vector space $V$, such that $S \subseteq Cl(W,f)$. 

For any  finite dimensional subspace $W$ of a vector space $V$ there is a  subspace $\tilde {W}$ of  $V$, such that the  following conditions hold:
\begin{enumerate}
	\item 
	$\dim_F \tilde {W}$ is even; 
	\item
	$W \subseteq \tilde {W}$; 
	\item
the restriction of the form $f$ to $\tilde {W}$ is nondegenerate. 	
\end{enumerate}

Assume now, that $\dim_F \tilde{W}=n$. Then
$S \subseteq Cl(W,f) \subseteq Cl(\tilde{W},f)$
and $$ Cl(\tilde{W},f) = M_{2^{\frac{n}{2}}}(F).$$
Therefore, by the definition of a locally matrix algebra and the Steinitz number	corresponding to it, the Clifford algebra  $Cl(V,f)$ is locally matrix and $\mathbf{n}(Cl(V,f))=2^{\infty}$.	
\end{proof}

To obtain more examples we consider a generalization  of Clifford algebras.

Choose a positive integer $l>1$ that is coprime with characteristic of the ground field $F.$  Let $\xi\in F$ be an $l$-th primitive root of $1.$ Consider the generalized Clifford algebra  \begin{multline*}
Clg(l,m)=\langle x_1,\ldots, x_m \, | \,  x^l_i =1; x_i^{-1}x_j x_i =\xi x_j  \text{ for } i<j; \\ x_i^{-1}x_j x_i =\xi^{-1} x_j  \text{ for } i>j; \ 1\leq i,j \leq m \rangle.
\end{multline*}  Such algebras in a more general form were considered in \cite{Ramakr}.

Similarly, for an arbitrary ordered set $I$ we consider  
\begin{multline*}
	Clg(l,I)=\langle x_i, i \in I \, | \,  x^l_i =1; x_i^{-1}x_j x_i =\xi x_j  \text{ for } i<j; \\ x_i^{-1}x_j x_i =\xi^{-1} x_j  \text{ for } i>j;\ i,j \in I \rangle.
\end{multline*} 

\begin{theorem}
\label{th6_1} \ \
 1) For an even $m$ we have $$Clg(l,m)\cong M_{l^{\frac{m}{2}}}(F);$$ 
	
	2) for an odd $m$ we have $Clg(l,m)\cong \underbrace{M_{l^{\frac{m-1}{2}}}(F)\oplus \cdots \oplus   M_{l^{\frac{m-1}{2}}}(F)}_l.$
\end{theorem}
\begin{proof} Using the Groebner--Shirshov technique  \cite{Bokut} we see that ordered monomials $x_1^{k_1}\cdots\cdots x_m^{k_m},$ $0\leq k_i\leq l-1,$ $1\leq k_i\leq m,$ form a basis of the algebra $Clg(l,m).$ In particular, $$\dim_F Clg(l,m)=l^m.$$
		
	For an arbitrary $1\leq i\leq m$ the mapping $$\varphi_i (x_j)=\xi^{\delta_{ij}} x_j, \ \ 1\leq j\leq m, $$ extends to an automorphism of the algebra $ Clg(l,m).$

	\begin{lemma}\label{lem_dop}  
		Let $V$ be a subspace of $Clg(l,m)$ that is invariant under $\text{Aut } Clg(l,m).$ Then $V$ is spanned by all ordered monomials lying in $V.$
	\end{lemma}
	\begin{proof} Let $v= \Sigma_{k=0}^{l-1} x_i^k  v_k \in V,$ where $v_k$ lies in the subalgebra generated by $x_1,$ $\ldots,$ $x_{i-1},$ $x_{i+1},$ $\ldots,$ $x_m.$ Then  $$\varphi_i^s (v)=\sum_{k=0}^{l-1}\xi^{sk} x_i^k v_k\in V, \ \ s=0,1,\ldots,l-1. $$ Viewing these $l$ equalities as a system of equations in $v_0,$ $x_i v_1,$ $\ldots,$  $x_i^{l-1} v_{l-1}$ with the Vandermonde determinant  $\det |\xi^{ij}|\neq 0$ we represent each $x_i^k v_k$  as a linear combination of $v,$ $\varphi_i(v),$ $\ldots,$ $\varphi_i^{l-1}(v).$  Hence each $x_i^k v_k$  lies in $V.$
		
		Repeating this argument several times we get the assertion of lemma. 
	\end{proof}
	
	\begin{corollary}\label{coll_1} 
		The algebra $Clg(l,m)$ does not contain proper ideals that are invariant under $\text{Aut }Clg(l,m).$ 
	\end{corollary}

	\begin{corollary}\label{coll_2} 
		The algebra $Clg(l,m)$ is semisimple.
	\end{corollary}
	
	Now our aim is  to find monomials  $x_1^{k_1}\cdots x_m^{k_m}$ that lie in the center $C$ of $Clg(l,m).$  By Lemma \ref{lem_dop} these  monomials span $C.$ Let $v=x_1^{k_1} x_2^{k_2} \cdots x_m^{k_m}\in C.$ We have $$ x_i^{-1} v x_i = \xi^{-k_1-k_2-\cdots -k_{i-1}+k_{i+1}+\cdots + k_m} v  . $$  Hence 
	\begin{equation}\label{eq_last}
	\begin{cases}
				k_2+\cdots + k_m & = 0  \mod l \\
		-k_1+k_3+\cdots + k_m & = 0  \mod l \\
		\ldots \ldots &  \ldots \ldots \\
		-k_1-k_2-\cdots - k_{m-1} & = 0  \mod l \\
		\end{cases}
		\end{equation}
		If $m$ is even then the $m\times m$ matrix $$K=\left(
		\begin{array}{cccccc}
		0 & 1 & 1 & \ldots & 1 & 1 \\
		-1 & 0 & 1 & \ldots & 1 & 1 \\
		\vdots  &  &  &  & & \\
		-1  & -1 & -1 & \ldots & 0 & 1 \\
		-1 & -1 & -1 & \ldots & -1 & 0 \\
		\end{array}
		\right)$$
		over $\mathbb{Z} / l \mathbb{Z}$ is invertible. Therefore in this case $v=1$ and $C=F\cdot 1.$ This proves the part 1) of the theorem. 
		
		If $m$ is odd then the rank of $K$ is $m-1$ and the system (\ref{eq_last}) has  $l$ solutions $(i,-i,i,-i,\ldots,i),$  $i=0,1,\ldots, l-1.$  Hence $\dim_F C=l. $ This implies that the algebra $Clg(l,m)$  is a direct sum of $l$ simple summands,  $Clg(l,m)=A_1 \oplus \cdots \oplus  A_l. $  Automorphisms of $Clg(l,m)$ permute the ideals $A_1,$ $\ldots,$ $A_l .$ By Corollary \ref{coll_1} of the Lemma \ref{lem_dop} they are conjugate under $\text{Aut }Clg(l,m),$ hence isomorphic. Since $\dim_F Clg(l,m)=l\cdot \dim_F A_1 =l^m ,$ it follows that $\dim_F A_1 =l^{m-1} ,$ $A_1 \cong  M_{l^{\frac{m-1}{2}}} (F).$ This proves the part 2) of the theorem.
	\end{proof}
\begin{theorem}
	\label{GenCl}
	The Steinitz number of a unital locally matrix algebra   $Clg(l,I)$, where the set $I$ is infinite,  is $l^{\infty }$.	
\end{theorem}

\begin{proof}
	For an arbitrary finite subset $S$ of $Clg(l,I)$ there exists a finite subset $J \subset  I$,  such that $S \subset  Clg(l,J).$  Without loss of generality, we can assume that $\mathfrak{Card}  J=m$ is an even number.Then by Theorem \ref{th6_1}  $Clg(l,J) \cong  M_{l^{\frac{m}{2}}} (F) .$ This completes the proof of the theorem.	
	\end{proof}

\begin{theorem}
	If $\tau=p_1^{k_1}p_2^{k_2} p_3^{k_3} \ldots$ is a Steinitz number, such that $k_j=\infty$ for some positive integer $j$, then for any infinite dimension $\alpha$ there is a unital locally matrix algebra $A$, such that $\dim_F A= \alpha$ and $\mathbf{n}(A)=\tau$.	
\end{theorem} 


\begin{proof}
Let $k_j=\infty$. Let   $\tau '=p_1^{k_1}p_2^{k_2}\ldots p_{j-1}^{j-1} p_{j+1}^{j+1} \ldots$. In \cite{Sushch2} it is proved that there exists a countable dimensional locally matrix algebra $A'$, such that $\mathbf{n}(A')=\tau'$.

Let $I$ be an  ordered set of cardinality $\alpha$, let  $A''=Clg(p_j,I)$. Then   $\mathbf{n}(A'')=p_j^{\infty}$ and $\mathbf{n}(A'\otimes A'')= \tau.$ It is easy to see that $\dim_F A'\otimes_F A''=\alpha$.	
\end{proof}

A Steinitz number $\tau=\prod_{p\in \mathbb{P}} p^{r_p} $ is called { \it locally finite}  if  $r_p < \infty $ for any $p \in  \mathbb{P} $.  

\begin{conjecture}
	If $\tau$ is a locally finite Steinitz number, $A$ is a unital locally matrix algebra and $\mathbf{n}(A)=\tau$,  then $A$  is countable dimensional.	
\end{conjecture}

\section{Non isomorphic  algebras with equal   Steinitz numbers}

We will  construct two  nonisomorphic unital locally matrix algebras of uncountable dimension, such  their  Steinitz numbers are equal.

Thus we show that Theorem \ref{univ1} in the case of  uncountable dimension is not true. Consider the vector space: 
$$V=\{(a_1, a_2, \ldots) \ | \ a_i \in \mathbb{C}, \sum_{i=1}^{\infty}|a_i|^2< \infty \}.$$ 
Let  $f$ be the  quadratic form:
$$f((a_1, a_2, \ldots))= \sum_{i=1}^{\infty}a_i^2 \in \mathbb{C}.$$
The  vector space $V$ has  uncountable dimension.

Assume that  $I$ is a set of indexes, whose  cardinality $ \mathfrak{Card}I$ is equal to the dimension of the vector space $V$. Let now $W$ be a  complex vector space with basis $w_i, i \in I$. Assume that $g$ is the quadratic form on $W$ determined for arbitrary $w=\alpha_1w_{i_1}+\alpha_2w_{i_2}+ \ldots+\alpha_nw_{i_n}$, $\alpha_i \in \mathbb{C}$  by the rule: 
$$g(w)= \sum_{i \in 1}^n\alpha_i^2.$$

Define Clifford algebras  $A=Cl(V,f)$ and   $B=Cl(W,g)$. As follows from their constructions, algebras  $A$ and $B$ are unital locally matrix and 
$$\dim_F A= \dim_F B=  \mathfrak{Card}  I.$$
In addition, from Theorem \ref{dimCl} it follows, that $$\mathbf{n}(A)=\mathbf{n}(B)=2^{\infty}.$$

\begin{theorem}
	\label{nonizom}
	Clifford algebras $A$ and $B$ are not isomorphic. 	
\end{theorem}

To prove this theorem we need some lemmas. 

Let $V$ be a  vector space  with a  quadratic form $f$. Then there is the natural $\mathbb{Z}/ 2\mathbb{Z}$-gradation  of the Clifford algebra $Cl(V,f)$:
$$Cl(V,f)=Cl(V,f)_{\bar{0}}+Cl(V,f)_{\bar{1}},$$
where subspaces  $Cl(V,f)_{\bar{0}}$ and $Cl(V,f)_{\bar{1}}$ are spanned by all even $\underbrace{V \cdot V \cdot \ldots  \cdot V}_{2i}$ and all odd  $\underbrace{V \cdot V \cdot \ldots \cdot V}_{2i+1}$ products respectively.

\begin{lemma}
	\label{chain}
	Let $V=V_0 \supset V_1 \supset V_2 \supset \ldots$ be a	descending chain of subspaces $dim_{\mathbb{C}}(V/ V_i) < \infty$, $i \ge 0$,  and $\bigcap_{i \ge 0} V_i = (0)$. Then  $$ \bigcap_{i \ge 0}Cl( V_i) = \mathbb{C} \cdot 1. $$
\end{lemma}

\begin{proof}
	Let us show that without loss of generality we can assume the space $V$ to be countable dimensional. Indeed, let $$a \in \bigcap_{i \ge 0} Cl(V_i), \qquad a \notin \mathbb{C} \cdot 1 .$$ For an arbitrary ${i \ge 0}$ there exists a finite dimensional subspace $W_i \subset V_i$ such that $a\in Cl(W_i)$. Let $V'=\sum_{i \ge 0} W_i$, $dim_{\mathbb{C}}V'=\varkappa_0$, and let $V'_i=V_i \cap V'$. Then 
	$V'=V'_0 \supset V'_1 \supset \ldots $,   $\bigcap_{i \ge 0} V'_i = (0)$ and $ \bigcap_{i \ge 0}Cl( V'_i) \ne \mathbb{C} \cdot 1 $.
	
	Assume therefore that the space $V$ is countable dimensional. Let $U=\{ u_0, u_1, \ldots \}$ be a basis of $V$. 
	
	Without loss of generality we assume that $u_0 \notin V_1$. 
	
	For any finite dimensional subspace $V' \subset V$ there clearly exists $n \ge 1$ such that $V' \cap V_n=(0).$ We will construct an ascending chain of finite subsets $U_0 \subset U_1 \subset \ldots$ of the basis $U$ and an increasing sequence of integers $n_1 < n_2 < \ldots$ such that
\begin{enumerate}[(1)]
		\item 
		\label{bas_1} 
		$u_0, u_1, \ldots, u_k \in U_k$,
		
	\item 
	\label{bas_2}
	$U_k$ is a basis of the space $V$ modulo $V_{n_{k+1}}$, for any $k \ge 0$.	
\end{enumerate}

Let $U_0$ be a maximal subset of $U$ such that $U_0$ is linearly independent modulo $V_1$ and $u_0 \in U_0$. Let $n_1=1$.

Suppose that subsets $U_0 \subset U_1 \subset \ldots U_k \subset U$ and integers $1=n_1< \ldots <n_{k+1}$ have been selected; the subset $U_k$ is a basis of $V$ modulo $V_{n_{k+1}}$; $u_0, \ldots, u_k \in U_k$.

Let $l \ge k+1$ be a smallest integer such that $u_l \notin U_k$; $u_0, u_1, \ldots, u_{l-1} \in U_k$. There exists an integer $n_{k+2}> n_{k+1}$ such that $U_k \cup \{u_l\}$ is linearly independent modulo $V_{n_{k+2}}$. Let $U_{k+1}$ be a maximal subset of $U$ that is linearly independent modulo $V_{n_{k+2}}$ and $U_k \cup \{u_l\} \subseteq U_{k+1}$. Clearly, $U_{k+1}$ is a basis of $V$ modulo $V_{n_{k+2}}$ and  $u_0, \ldots, u_{k+1} \in U_{k+1}$. Considering the subspaces $V_{n_{k}}$ instead of $V_k$, $k \ge 1$, we can assume that $U_k$ is a basis of $V$ modulo $V_{k+1}$, $k \ge 0$.

	Now we will construct a new ordered basis of $V$. Let $B_0=U_0$. Let $U_1=U_0 \cup \{w_1, \ldots, w_k\}$. For each $w_i$, $1 \le i \le k$, there exists an element $w'_i \in Span U_0$, such that $(w_i-w'_i) \in V_1$. Let $$B_1=\{ w_1-w'_1, \ldots,     w_k-w'_k \} \subset V_1.$$ Clearly, $Span(B_0 \cup B_1)=Span U_1.$ Continuing in this way we construct disjoint finite set $B_1, B_2, \ldots $ such that $B_n \subset V_n$ and $Span(B_0 \cup B_1 \cup \ldots B_n)=Span U_n$, $n \ge 0$. Hence $B=\bigcup _{i \ge 0}B_i$ is a basis of $V$.
	
	Let all elements in $B_j$ be greater than all elements in $B_i$, $i<j$. For each $i \ge 0$ elements in $B_i$ are ordered in an arbitrary way. 
	
 Strictly ordered products of elements from $B$ form a basis  of the Clifford algebra  $Cl( V,f)$. Ordered products $b_1b_2 \ldots b_t$, where $b_1, \ldots , b_t \in B \cap V_i$, form a basis of $Cl( V_i,f)$. This implies the assertion of the Lemma. 
\end{proof}

Let now $v$ be an element of $V$, such that $f(v) \neq 0$. Define a set 
$$v^{\perp}= \{v' \in V \ | \ f(v,v')=0\}.$$
It is clear, that  
\begin{equation}
\label{ortogon}
V=\mathbb{C} v + v^{\perp}. 
\end{equation} 

By  $C(v)$ we denote the centralizer of the element  $v$ in the Clifford algebra  $Cl(V,f)$, i.e. $C(v)= \{a \in Cl(V,f) \ | \ v\cdot a=a\cdot v \}$.

\begin{lemma}
	\label{lema 2.2}
	  
	\begin{enumerate}[(A)]
		\item 
		\label{centr_ort_1}
		If $v \in V$, $f(v)=1, $ then $$	C(v)=v Cl(v^{\perp},f)_{\bar{0}} +  Cl(v^{\perp},f)_{\bar{0}}.$$
		
		\item 
		\label{centr_ort_2}
		Let $n$ be an odd number. Then for the elements  $v_i=(a_{i1},a_{i2}, \ldots, a_{ii},\ldots)$, $1 \le i \le n$, such that   $a_{ij}=\begin{cases}
		1, \text{if  $j=i$}\\
		0, \text{if $j \ne i$}
		\end{cases},$    we have 
		\begin{multline*}
		C(v_1) \cap C(v_2) \cap \ldots \cap  C(v_n)= \\= v_1 v_2 \ldots v_n Cl(v_1^{\perp} \cap v_2^{\perp} \cap \ldots \cap v_n^{\perp},f)_{\bar{0}}+ 
		Cl(v_1^{\perp} \cap v_2^{\perp} \cap \ldots \cap v_n^{\perp},f)_{\bar{0}}.  \end{multline*}			
	\end{enumerate} 	
\end{lemma}

\begin{proof}
	The equality \eqref{ortogon} implies, that any element from $Cl(V,f)$ is equal to a  sum of products of the element   $v$ and some elements from $v^{\perp}$.
	
	Define the operator 	$U_v: Cl(V,f) \to Cl(V,f) $, via $x \mapsto vxv $. As $f(v)=1$ we have
	$$U_v^2(x)= v v xvv=x.$$
	Hence the operator $U_v^2 $ has two eigenvalues $-1$ and $1$.
	For any vector $u$ from $v^{\perp}$ we have $$uv+vu=f(u,v) \cdot 1=0.$$
	So $uv=- vu$. As $f(v)=1, $ we obtain the equality  $$U_v(u)=vuv=-uvv=-u.$$
	Hence any vector from $v^{\perp}$ is an eigenvector of the operator  $U_v$ with the eigenvalue $-1$. 
	
	Let now $u_1, u_2, \ldots, u_n$ be elements from the set $v^{\perp}$. For the reasons given above it follows that if $n$ is even, then  $u_1  u_2  \ldots  u_n$ is an eigenvector  of the operator  $U_v$ with the eigenvalue $1$.
	Therefore the element  $ u_1  u_2  \ldots  u_n$ of $Cl(V,f)$ is an element of the centralizer $C(v)$. Note, that  if $n$ is odd, then  $u_1  u_2  \ldots  u_n$ is an eigenvector  of the operator  $U_v$ with the eigenvalue $-1$, and so $u_1  u_2  \ldots  u_n$ is not an element of $C(v)$.  
	
Similarly, if  $n$ is even, then  $vu_1  u_2  \ldots \cdot u_n$ is an eigenvector  of the operator  $U_v$ with the eigenvalue $1$ and   $ v u_1  u_2  \ldots  u_n$ lies in $Cl(V,f)$ in the centralizer $C(v)$. If  $n$ is odd, then  $v u_1  u_2  \ldots u_n$  is not an element of $C(v)$.    Therefore $$	C(v)=v Cl(v^{\perp},f)_{\bar{0}} +  Cl(v^{\perp},f)_{\bar{0}},$$ 
	and the first statement is proved.
	
	As for any $i$ we have $f(v_i)=1$, the proof of the second statement follows  from the first statement by induction.  
\end{proof}  

Let $\mathcal{V}$ be the set $\{v_i=(\underbrace{0, \ldots, 0, 1,}_{i}0,\ldots), i \ge 1 \} $. 

\begin{proposition}
	\label{prop2_non}
	The centralizer of the set $\mathcal{V} $ in the algebra $A=Cl(V,f)$ is $\mathbb{C}\cdot 1.$
\end{proposition}

\begin{proof}
	Consider a centralizer of the set $\mathcal{V}:$  $$C(\mathcal{V})=\bigcap_{i\geq 1} C(v_i).$$
	The subspace $C(\mathcal{V})$ of the algebra $Cl(V,f)$  is  $\mathbb{Z}/2\mathbb{Z}$--graded, i.e. $$C(\mathcal{V})=C_{\overline{0}}+C_{\overline{1}}.$$
	
Assume, that  $a\in C_{ \bar{1} }$. Then   
\begin{equation}
\label{rozklad}
 a \in \sum_{i=0}^{m}\underbrace{V \cdot V \cdot \ldots  \cdot V}_{2i+1}.
 \end{equation} Let $n$ be an odd number such, that $2m+1<n.$ By Lemma \ref{lema 2.2}, \eqref{centr_ort_2} $a \in v_1 v_2 \ldots v_n Cl(v_1^{\perp} \cap v_2^{\perp} \cap \ldots \cap v_n^{\perp}, f)_{\bar{0}}$. Let $\{w_j\}_{j \in J}$, be an ordered basis of the subspace $v_1^{\perp} \cap v_2^{\perp} \cap \ldots \cap v_n^{\perp}$. Then $v_1, \ldots, v_n, w_j$, $j \in J$ is a basis of $V$. The set of ordered products of elements of this basis is a basis of $Cl(V,f)$. Elements  $v_1 \ldots v_n w_{j_1}, \ldots w_{j_k}$, $j_1<\ldots< j_k$, lie in this basis. Hence, $$v_1 \ldots v_n Cl(v_1^{\perp} \cap v_2^{\perp} \cap \ldots \cap v_n^{\perp},f)\cap (\sum_{i=0}^{n-1} \underbrace{V\ldots V}_{i})=(0).$$ Therefore, $a=0$ and $C_{\overline{1}}=(0).$ So, $ C(\mathcal{V})=C_{\overline{0}}$. According to Lemma~\ref{lema 2.2},~\eqref{centr_ort_2} $C_{\overline{0}}= \cap_{i\geq 0} Cl(v_i, f)_{\overline{0}} $.    Thus  Lemma~\ref{chain} implies $C(\mathcal{V})= \mathbb{C}\cdot 1.$
\end{proof}

\begin{proposition}
	\label{prop3_non_is}
	For an arbitrary countable subset $X$ of the algebra $B=Cl(W,g)$ the centralizer of  $X$ in $B$ is different from $\mathbb{C}\cdot 1.$
\end{proposition}

\begin{proof}
	Let $X$ be a countable subset of the algebra $B=Cl(W,g).$    There exists a countable subset $I_0\subset I$, such that $X\subset Cl(W_0,g),$ where $W_0=\text{Span}(w_i, i\in I_0).$ So, for any indexes $i,$ $j \in I\setminus I_0,$ $i\neq j,$ the element $w_iw_j$ belongs the centralizer of $X.$
\end{proof} 

\begin{proof}[Proof of Theorem \ref{nonizom}]
By Proposition \ref{prop2_non} the centralizer of the set $\mathcal{V} $ in the Clifford algebra $A$ is $\mathbb{C}\cdot 1.$	But by Proposition \ref{prop3_non_is} the centralizer of  $X$ in the Clifford algebra $B$ is different from $\mathbb{C}\cdot 1$ for an arbitrary countable subset $X$ of $B$. Therefore, 	$A$ and $B$ are not isomorphic.   
\end{proof} 

\section{Acknowledgement}
We are  grateful to E. Zelmanov for helpful discussions and   valuable comments.

\end{document}